\documentclass[12pt]{amsart}
\usepackage{amssymb,amsmath,amsfonts,amsthm,bbm,mathrsfs,times,graphicx,color,comment,mathabx,enumerate}
\usepackage{bigints}
\usepackage[utf8]{inputenc}
\usepackage{tikz}
\usetikzlibrary{arrows,shapes}
\usetikzlibrary{fadings}
\usetikzlibrary{intersections}
\usetikzlibrary{decorations.pathreplacing}

\usepackage{mathtools}

\DeclareMathOperator{\re}{\mathbb{R}e}
\DeclareMathOperator{\im}{\mathbb{I}m}

\newcommand{\norm}[1]{\left\|#1\right\|}

\renewcommand{\i}{\mathrm i}

\newcommand{\INF}{{\infty}}

\newcommand{\dist}{\mbox{dist}}
\newcommand{\tta}{\theta}

\newcommand{\OM}{\Omega}

\newcommand{\sph}{{{\mathbf S}^ 1}}

\newcommand{\del}{\partial}

\newcommand{\Gam}{\varGamma}
\newcommand{\ol}{\overline}
\newcommand{\ds}{\displaystyle}
\newcommand{\dba}{\overline{\partial}}

\newcommand{\BR}{\mathbb{R}}
\newcommand{\BC}{\mathbb{C}}

\newcommand{\bu}{{\bf u}}
\newcommand{\bv}{{\bf v}}

\newcommand{\bzero}{\mathbf 0}

\newcommand{\btheta}{\boldsymbol \theta}
\newcommand{\balpha}{\boldsymbol \alpha}

\newcommand{\B}{\mathcal{B}}

\newcommand{\HT}{\mathcal{H}}

\newcommand{\lnorm}[1]{ \left\| #1 \right\|}

\newtheorem{theorem}{Theorem}[section]
\newtheorem{prop}{Proposition}[section]
\newtheorem{lemma}{Lemma}[section]

\title[A source reconstruction method]{A  source reconstruction method in two dimensional radiative transport using boundary data measured on an arc}

\begin{document}
\date{\today}
\author{Hiroshi Fujiwara}
\address{Graduate School of Informatics,  Kyoto University, Yoshida Honmachi, Sakyo-ku, Kyoto 606-8501, Japan }
\email{fujiwara@acs.i.kyoto-u.ac.jp}
\author{Kamran Sadiq}
\address{Johann Radon Institute of Computational and Applied Mathematics (RICAM), Altenbergerstrasse 69, 4040 Linz, Austria}
\email{kamran.sadiq@ricam.oeaw.ac.at}
\author{Alexandru Tamasan}
\address{Department of Mathematics, University of Central Florida, Orlando, 32816 Florida, USA}
\email{tamasan@math.ucf.edu}

\subjclass[2010]{Primary 35J56, 30E20; Secondary 45E05}
\keywords{Radiative transport, source reconstruction, scattering, $A$-analytic maps, Hilbert transform, Bukhgeim-Beltrami equation, optical tomography, optical molecular imaging}
\maketitle

\begin{abstract}
We consider an inverse source problem in the stationary radiating transport through a two dimensional absorbing and scattering medium. Of specific interest, the exiting radiation is measured on an arc. The attenuation and scattering properties of the medium are assumed known. For scattering kernels of finite Fourier content in the angular variable, we show how to quantitatively recover the part of the isotropic sources restricted to the convex hull of the measurement arc.  The approach is based on the Cauchy problem with partial data for a Beltrami-like equation associated with $A$-analytic maps in the sense of Bukhgeim, and extends authors' previous work to this specific partial data case. 
The robustness of the method is demonstrated by the results of several numerical experiments.  
\end{abstract}
\section{Introduction}

Let $\Omega$ be a strictly convex planar domain. In the steady state case, when generated solely by a source $f$ inside $\OM$,  the density  $u(z,\btheta)$ of particles at $z$ traveling in the direction $\btheta$  through an absorbing and scattering domain $\Omega$ solves the stationary transport boundary value problem 
\begin{equation}\label{TransportScatEq1}
\begin{aligned} 
&\btheta\cdot\nabla u(z,\btheta) +a(z,\btheta) u(z,\btheta) -  \int_{\sph} k(z,\btheta,\btheta')u(z,\btheta') d\btheta' = f(z,\btheta) , \quad (z,\btheta)\in \OM\times\sph,\\
&u\lvert_{\Gamma_-}=0,
\end{aligned}
\end{equation}
where $\Gam_{-} :=\{(\zeta,\btheta)\in \Gam \times\sph:\, \nu(\zeta)\cdot\btheta<0 \}$
with $\nu$ being the outer unit normal field at the boundary. The boundary condition indicates that no radiation is coming from outside the domain. Throughout, the measure on the circle is normalized to $ \int_{\sph} d\btheta =1$. 

The boundary value problem  \eqref{TransportScatEq1} is known to have a unique solution under various ``subcritical" assumptions, e.g., \cite{dautrayLions4, choulliStefanov96, choulliStefanov99, anikonov02,mokhtar}, with a general result in  \cite{stefanovUhlmann08} showing that, for an open and dense set of coefficients $a\in C^2(\ol\Omega\times\sph)$ and $k\in C^2(\ol\Omega\times\sph\times\sph)$, 
the problem  \eqref{TransportScatEq1} has a unique solution $u\in L^2(\Omega\times\sph)$ for any $f\in L^2(\Omega\times\sph)$. Some of our arguments in the reconstruction method here require solutions $u\in C^{1,\mu}(\ol\OM\times\sph)$, $\frac{1}{2}<\mu<1$. We revisit the arguments in  \cite{stefanovUhlmann08} and show that such a regularity can be achieved for sources $f\in W^{2.p}(\Omega\times\sph)$, $p>4$; see Theorem \ref{u_reg_Wp} (ii) below.

For an arc $\Lambda$ of the boundary of $\OM$, we consider the inverse problem of the determining $f$ from measurements $g$ of exiting radiation on $\Lambda$,
\begin{align}\label{data}
u|_{\Lambda_+} = g,
\end{align}
where  $\Lambda_+:=\{(z,\btheta)\in \Lambda \times\sph:\,\nu(z)\cdot\btheta>0 \}$ with $\nu$ being the outer unit normal field at the boundary. This is an inverse source problem  with partial boundary data, which applies to multiple anisotropic scattering media of non-small anisotropy far from diffusive but also single-scattering regime with applications to Optical Molecular Imaging and Optical Tomography \cite{arridge,schotland}.


Except for the results in Section 2, which concern the forward problem, in this work the source and attenuation coefficient are assumed isotropic, $f=f(z)$ and  $a=a(z)$, and that the scattering kernel $k(z,\btheta,\btheta')=k(z,\btheta\cdot\btheta')$ depends polynomially on the angle between the directions,
\begin{align}\label{particular_kM}
 k(z,\cos \tta)  = k_0 +2\sum_{n=1}^{M} k_{-n}(z) \cos (n \tta), 
\end{align} 
for some fixed integer $M \geq1$. Moreover, the functions $a,f,k$ are assumed real valued. These assumptions occur naturally in radiative transfer models in Optics, see, e.g., \cite{chandrasekhar}.

Our main result, Theorem \ref{main}, shows that  $u|_{\Lambda_+}$ determines both $f$ and $u$ in the convex hull of $\Lambda$, and provides a method of reconstruction.  Specific to two dimensional domains, our approach is based on the Cauchy problem with partial data for a Beltrami-like equation associated with $A$-analytic maps in the sense of Bukhgeim \cite{bukhgeimBook}, and extends the authors' previous work \cite{fujiwaraSadiqTamasan19}, which used measurements on the entire boundary,  to this specific partial data case. The new insight is that, similar to the non-scattering case \cite{fujiwaraSadiqTamasan21}, the trace $u|_{\Lambda}$ determines $u|_L$ provided the scattering kernel has finite Fourier content as in \eqref{particular_kM}. The role of the finite Fourier content has been independently recognized in \cite{balMonard21}. 

We are aware of only one existing work that reconstructs the source in the presence of scattering using partial boundary data in stationary radiative transport. Namely, in \cite{smirnovKlibanovNguyen19} the domain is a slab and data is available on each side on sufficiently long intervals. In contrast, our method here assumes only ``one sided" boundary data, and does not require iterative solvability of the forward problem.

As demonstrated by the numerical experiments in Section \ref{Sec:numerics}, the method is robust in the sense that it reconstructs discontinuous sources, even when embedded in media with a discontinuous absorption property, and a scattering kernel of infinite Fourier content in the angular variable. 


\section{Remarks on the existence and regularity of the forward problem}
The well posedness in $L^p(\OM\times\sph)$ of the boundary value problem \eqref{TransportScatEq1} relies on the following compactness result, proven in \cite[Lemma 2.4]{stefanovUhlmann08} for the case $p=2$ and $a$ and $k$ twice differentiable. In this section we revisit the arguments in \cite{stefanovUhlmann08} for any $1<p<\infty$, and show that they hold if the attenuation is merely \emph{once} differentiable. We work in two dimensions but this is not essential. Adopting the notation in \cite{stefanovUhlmann08}, let us consider the operators
\begin{equation} \label{Tinv_K_defn}
\begin{aligned}
[T_1^{-1} g ] (x, \btheta) &= \int_{-\INF}^{0} e^{- \int_{s}^{0} a(x +t \btheta, \btheta ) dt} g(x + s \btheta,\btheta) ds,
\quad \text{and} \\
[K g ] (x, \btheta) &= \int_{\sph} k(x,\btheta,\btheta')g(x,\btheta') d\btheta',
\end{aligned}
\end{equation}where the intervening functions are extended by 0 outside $\OM$.

Using the formal expansion
\begin{equation}\label{u_decomp}
\begin{aligned}
 u = T_1^{-1} f&+ T_1^{-1} K T_1^{-1} f +T_1^{-1} (KT_1^{-1}K)[I - T_1^{-1} K]^{-1} T_1^{-1} f,
\end{aligned}
\end{equation}
the well posed-ness in $L^p(\Omega\times\sph)$ of the boundary value problem \eqref{TransportScatEq1} reduces to the invertibility of $I- T_1^{-1}K$  in  $L^p(\OM\times \sph)$.

To further simplify notations, let $\widehat{x}=\frac{x}{|x|}$, so that for $y\neq x$ we have $y=x - |x-y|(\widehat{x-y})$.
\begin{prop}\label{KToneK_prop}
 Let $a\in C^1(\ol\Omega\times\sph)$ and $k \in C^{2}(\ol\Omega\times\sph\times\sph)$. Then the operator
\begin{align} \label{KToneK_reg}
 K T_1^{-1} K: L^p(\OM \times \sph) \longrightarrow W^{1,p}(\Omega\times\sph) \mbox{ is bounded},
\; 1<p<\infty. 
\end{align}
\end{prop}   
\begin{proof}
Using the definitions of $T_1^{-1}$ and $K$ above, and a change to polar coordinates, one can write
\begin{align}\label{KTinvK_defn}
 [K T_1^{-1} K g] (x, \btheta) =  \int_{\sph}\int_{\OM} \frac{ \eta \left( x, |x-y|,(\widehat{x-y}),  \btheta, \btheta' \right) }{|x-y|}  g \left( y, \btheta' \right )  dy d \btheta',
 \end{align}
 where, for $\left( x, r,\balpha, \btheta, \btheta' \right)\in\OM\times[0,\infty)\times\sph\times\sph\times\sph$,
\begin{align*}
 \eta \left( x, r,\balpha, \btheta, \btheta' \right)  = 
  e^{- \int^{r}_{0} a \left(x -t \balpha,\balpha  \right) dt}
 k \left(x, \btheta, \balpha \right) k \left(x-r\balpha, \balpha,\btheta' \right).
\end{align*}

An application of the fundamental theorem of calculus to $r\mapsto \eta \left( x, r,\balpha, \btheta, \btheta' \right)$ yields
\begin{align}\label{k_split}
\displaystyle \eta \left( x, r,\balpha, \btheta, \btheta' \right)  =  \eta \left( x, 0,\balpha, \btheta, \btheta' \right)  - r k(x,\btheta,\balpha)\eta_1 \left( x, r,\balpha, \btheta' \right),\end{align}
where, for $\balpha=(\alpha_1,\alpha_2)$, the function $\eta_1(x,r,\balpha,\btheta')$ is defined by
\begin{align*}
\int_0^1 e^{-\int_0^{r\rho} a(x-t\balpha,\balpha)dt}
\left[a(x-r\rho\balpha, \balpha) k(x-r\rho\balpha,\balpha,\btheta') +\sum_{j=1}^2\alpha_j\frac{\partial k}{\partial x_j} (x-r\rho\balpha,\btheta')\right]d\rho.
\end{align*}Note that in $\eta_1$ there are no derivatives taken on $a$, whereas there are first order derivatives on $k$.

The split of the kernel in \eqref{k_split} induces the split of the operator $KT_1^{-1}K= A-B$ with
\begin{align*}
& [A g] (x, \btheta) =  \int_{\sph}\int_{\OM} \frac{ k \left(x, \btheta, \widehat{x-y} \right)  k \left(x, \widehat{x-y},\btheta' \right)
 }{|x-y|}  g \left( y, \btheta' \right )  dy d \btheta',\;\mbox{ and }\\
 &[B g] (x, \btheta) = \int_{\sph}\int_{\OM}  k \left(x, \btheta,\widehat{x-y} \right)\eta_1 \left( x, |x-y|,\widehat{x-y}, \btheta' \right) g \left( y, \btheta' \right )  dy d \btheta'.
 \end{align*}

Using 
\begin{align*}
\nabla\frac{1}{|x-y|}=\frac{\widehat{y-x}}{|x-y|^2},\quad
\frac{\del}{\del x_j}\left(\frac{x_k-y_k}{|x-y|}\right)=\frac{1}{|x-y|}\left(\delta_{jk}- \frac{(x_j-y_j)(x_k-y_k)}{|x-y|^2}\right),
\end{align*}and the regularity $a\in C^1(\ol\OM\times\sph)$ and $k\in C^2(\ol\OM\times\sph\times\sph)$, by a straightforward calculation of the derivatives, 
one can verify that $\frac{\partial}{\partial \theta_j}B$ is an operator with bounded kernel, while $\frac{\partial}{\partial x_j}B$ and $\frac{\partial}{\partial \theta_j}A$  are operators with weakly singular kernel, all of which are bounded in $L^p(\OM\times\sph)$, e.g., \cite[Theorem  VIII.3.1]{mikhlinProessdorf_book}.
On the other hand, in addition to the terms with weakly singular kernels, the derivatives $\frac{\partial}{\partial x_j} A $ also yield operators of the Calder\'on-Zygmund type,
\begin{align}\label{CZop}
[\mathcal{C}v](x,\btheta)=\int_\sph\int_\OM\frac{\phi(x,\widehat{x-y},\btheta,\btheta')}{|x-y|^2}v(y,\btheta')dyd\btheta',
\end{align}
where the characteristic $\phi$ satisfies
\begin{align}\label{boundedChrt}
\sup_{\OM\times\sph\times\sph\times\sph}|\phi(x,\balpha,\btheta,\btheta')| <\infty.
\end{align}
The following lemma concludes the proof of the Proposition \ref{KToneK_prop}. \end{proof}
\begin{lemma}
Let $\mathcal{C}$ be the operator in \eqref{CZop} with the characteristic $\phi$ satisfying \eqref{boundedChrt}. Then
$\mathcal{C}$ is bounded in $L^p(\Omega\times\sph)$, $p>1$.
\end{lemma}
\begin{proof}
For brevity let us introduce the following notations
\begin{align}\label{g}
\phi_\infty(x,\balpha) :=\sup_{\btheta,\btheta'\in\sph}|\phi(x,\balpha,\btheta,\btheta')|,\quad g(y):=\int_\sph |v(y,\btheta')| d\btheta',
\end{align}and note that, via the H\"{o}lder inequality,
$\ds\|g\|_{L^p(\OM)}\leq \|v\|_{L^p(\OM\times\sph)}.$ 
Using  the Calderon-Zygmund boundedness theorem \cite[Theorem XI.3.1]{mikhlinProessdorf_book} (the third inequality below), we estimate
\begin{align*}
\|\mathcal{C}v\|^p_{L^p(\OM\times\sph)}&\leq \int_\OM\sup_{\btheta\in\sph}|[\mathcal{C}v](x,\btheta)|^p dx\leq \int_\OM\left|\int_\sph \int_\OM\frac{\phi_\infty(x,\widehat{x-y})}{|x-y|^2}|v(y,\btheta')|dyd\btheta'\right|^p dx\\
&= \int_\OM\left| \int_\OM\frac{\phi_\infty(x,\widehat{x-y})}{|x-y|^2}g(y)dy\right|^p dx\leq C\|g\|^p_{L^p(\OM)}
\leq C\|v\|^p_{L^p(\OM\times\sph)}.
\end{align*}
\end{proof}

The following simple result is useful.
\begin{lemma}\label{12}Let $X$ be a Banach space and $A:X\to X$ be bounded. Then $I-A$ has a bounded inverse in $X$ if and only if $I-A^2$ has a bounded inverse in $X$.
\end{lemma}
As a consequence, for $\lambda\in\BC$, the operator $I- T_1^{-1}(\lambda K)$ is invertible in  $L^p(\OM\times \sph)$ if and only if  $I- (T_1^{-1}(\lambda K))^2$ is invertible in  $L^p(\OM\times \sph)$.
By Proposition \ref{KToneK_prop}, $\ds (T_1^{-1} (\lambda K))^2$  is compact for any $z\in\BC$.  Since $\ds I- (T_1^{-1} (\lambda K))^2$ is invertible for $\lambda$ in a neighborhood of $0$, an application of the meromorphic Fredholm alternative (in reflexive Banach spaces, e.g., \cite{dunfordScwartz}) yields the following result.
\begin{theorem}\label{analytic_Fredholm} 
  Let $p>1$, $a\in C^{1}(\ol\Omega\times\sph)$, and $k\in C^{2}(\ol\Omega\times\sph\times\sph)$. At least one of the following statements is true.
  
  (i) $\ds I - T_1^{-1} K$ is invertible in $ L^p(\OM\times \sph)$.
    
  (ii) there exists $\epsilon>0$ such that $\ds I - T_1^{-1} (\lambda K)$ is invertible in $ L^p(\OM\times \sph)$, for any 
  $0< |\lambda-1|<\epsilon$.
  \end{theorem}

If $a\in C^2(\ol\OM\times\sph)$, then the regularity of the solution $u$ of \eqref{TransportScatEq1} increases with the regularity of $f$ as follows.
\begin{theorem} \label{u_reg_Wp}
 Consider the boundary value problem \eqref{TransportScatEq1} with $a\in C^2(\ol\OM\times\sph)$. For $p>1$, let $k\in C^{2}(\ol\Omega\times\sph\times\sph)$ be such that $\ds I - T_1^{-1}K$ is invertible in $L^p(\OM\times\sph)$, and let $u\in L^p(\OM\times\sph)$ in \eqref{u_decomp} be the solution  of \eqref{TransportScatEq1}.
 
 (i) If  $f \in W^{1,p}(\OM \times \sph)$, then $u \in W^{1,p}(\OM \times \sph)$. 
 
 (ii) If  $f \in W^{2,p}(\OM \times \sph)$, then $u \in W^{2,p}(\OM \times \sph)$. 
  
\end{theorem}
\begin{proof}

(i) Recall the representation \eqref{u_decomp} of the solution of \eqref{TransportScatEq1},
\begin{align*}
 u = T_1^{-1} f&+ T_1^{-1} K T_1^{-1} f +T_1^{-1} [KT_1^{-1}K ](I - T_1^{-1} K)^{-1} T_1^{-1} f.
\end{align*}
It is easy to see that  $T_1^{-1}$ and $K$  preserve the space $W^{1,p}(\Omega\times\sph)$, so that the first two terms  belong to $W^{1,p}(\Omega\times\sph)$.
Now, by Proposition \ref{KToneK_prop}, the third term is also in $W^{1,p}(\OM\times\sph)$. 

(ii)  For brevity we introduce the operators
\begin{equation}\label{T0tilde_Kj_defn}
\begin{aligned}
&T_0^{-1}u(x,\btheta)=\int_{-\infty}^0 u(x+t\btheta,\btheta)dt, \quad K_ju(x,\btheta)=\int_\sph \frac{\del k}{\del x_j}(x,\btheta,\btheta')u(x,\btheta')d\btheta',  \\
&\widetilde{T}_0^{-1}u(x,\btheta)=\int_{-\infty}^0 u(x+t\btheta,\btheta)tdt, \quad 
\widehat{K}_ju(x,\btheta)=\int_\sph \frac{\del k}{\del \theta_j}(x,\btheta,\btheta')u(x,\btheta')d\btheta',\;j=1,2.
\end{aligned}
\end{equation}It is easy to see that $T_0^{-1}, \widetilde{T}_0^{-1},K_j$ and $\widehat{K}_j$ preserve $W^{1,p}(\OM\times\sph)$. 

By evaluating \eqref{TransportScatEq1} at $x + t \btheta$ and integrating in $t$ from $-\INF$ to $0$, the problem \eqref{TransportScatEq1} with zero incoming fluxes is equivalent to the integral equation.

\begin{align}\label{TransportScatEq1_Equiv}
u+ T_0^{-1} (au)- T_0^{-1}Ku=T_0^{-1}f.
\end{align}

For $f\in W^{1,p}(\OM\times\sph)$, according to part (i),  $u_{x_j}\in L^p(\OM\times\sph)$. In particular $u_{x_j}$ solves the integral equation
\begin{align}\label{TransportScatEq1_Equiv_delxU}
u_{x_j} + T_0^{-1} (au_{x_j})- T_0^{-1}Ku_{x_j}=T_0^{-1}f_{x_j} - T_0^{-1} (a_{x_j}u)+ T_0^{-1}K_ju
\end{align}Moreover, since $a\in C^2(\ol\OM\times\sph)$, $k\in C^{2}(\ol\Omega\times\sph\times\sph)$, and $f \in W^{2,p}(\OM \times \sph)$, the right-hand-side of  \eqref{TransportScatEq1_Equiv_delxU} lies in $W^{1,p}(\OM \times \sph)$. By applying part (i) above, we get that the unique solution to \eqref{TransportScatEq1_Equiv_delxU}
\begin{align}\label{xfirst}
u_{x_j}  \in W^{1,p}(\OM \times \sph), \; j=1,2.
\end{align}

For $f\in W^{1,p}(\OM\times\sph)$, also according to part (i),  $u_{\theta_j}\in L^p(\OM\times\sph)$. In particular $u_{\theta_j}$ is the unique solution of the integral equation
\begin{align}\label{TransportScatEq1_Equiv_DthetaU}
u_{\theta_j}+ T_0^{-1}(a u_{\theta_j}) = T_0^{-1} f_{\theta_j} -\widetilde{T}_0^{-1} (au_{x_j})-\widetilde{T}_0^{-1} (a_{x_j}u) - T_0^{-1}(a_{\theta_j}u) +\widetilde{T}_0^{-1}K_ju- T_0^{-1}\widehat{K}_j u,
\end{align} which is of the type \eqref{TransportScatEq1_Equiv} with $K=0$. Moreover, since $f \in W^{2,p}(\OM \times \sph)$, and, according to  \eqref{xfirst}, $u_{x_j}  \in W^{1,p}(\OM \times \sph), \; j=1,2$, the right-hand-side of  \eqref{TransportScatEq1_Equiv_DthetaU} lies in $W^{1,p}(\OM \times \sph)$. Again, by applying part (i), we get 
$$u_{\theta_j} \in W^{1,p}(\OM \times \sph), \; j=1,2.$$ Thus, $u \in W^{2,p}(\OM \times \sph)$.

\end{proof}

\section{Ingredients from  $A$-analytic theory}
For $0<\mu<1$, and  $\Gam$ (some part of) the boundary $\del \OM$ we consider the Banach spaces:
\begin{equation*}
 \begin{aligned} 
 l^{1,1}_{\INF}(\Gam) &:= \left \{ \bv\; : \lnorm{\bv}_{l^{1,1}_{\INF}(\Gam)}:= \sup_{\xi \in \Gam}\sum_{j=1}^{\INF}  j \lvert v_{-j}(\xi) \rvert < \INF \right \},\\
 C^{\mu}(\Gam; l_1) &:= \left \{ \bv:
\sup_{\xi\in \Gam} \lVert \bv(\xi)\rVert_{\ds l_{1}} + \underset{{\substack{
            \xi,\eta \in \Gam \\
            \xi\neq \eta } }}{\sup}
 \frac{\lVert \bv(\xi) - \bv(\eta)\rVert_{\ds l_{1}}}{|\xi - \eta|^{ \mu}} < \INF \right \}.
 \end{aligned}
 \end{equation*}
We similarly consider $l^{1,1}_{\INF}(\ol \OM)$, $C^{\mu}(\overline\OM ; l_1)$, and $C^{\mu}(\ol\OM ; l_{\INF})$.

A sequence valued map $\OM \ni z\mapsto  \bv(z): = \langle v_{0}(z), v_{-1}(z),v_{-2}(z),... \rangle$ in $C(\ol\OM;l_\INF)\cap C^1(\OM;l_\INF)$
is called {\em $\mathcal{L}^2$-analytic} (in the sense of Bukhgeim), if
\begin{equation}\label{Aanalytic}
\ol{\del} \bv (z) + \mathcal{L}^2 \del \bv (z) = 0,\quad z\in\OM,
\end{equation} where  $\mathcal{L}$ is the left shift operator,
$\displaystyle \mathcal{L} \langle v_{0}, v_{-1}, v_{-2}, \cdots  \rangle =  \langle v_{-1}, v_{-2},  \cdots \rangle,$
and $\mathcal{L}^{2}=\mathcal{L}\circ \mathcal{L}$. 
Note that we use the sequences of non-positive indexes to conform with the original notation in Bukhgeim's work \cite{bukhgeimBook}.

Analogous to the analytic maps,  the $\mathcal{L}^2$-analytic maps are determined by their boundary values via a Cauchy-like integral formula \cite{bukhgeimBook}. 
Following \cite{finch}, the Bukhgeim-Cauchy operator $\B$ acting on $\bv=\langle v_{0}, v_{-1}, v_{-2},...\rangle $
is defined component-wise for $n\leq 0$ by
\begin{align} \label{BukhgeimCauchyFormula}
(\B \bv)_{n}(z) &:= \frac{1}{2\pi \i} \int_{\Gam}
\frac{ v_{n}(\zeta)}{\zeta-z}d\zeta  + \frac{1}{2\pi \i}\int_{\Gam} \left \{ \frac{d\zeta}{\zeta-z}-\frac{d \ol{\zeta}}{\ol{\zeta-z}} \right \} \sum_{j=1}^{\infty}  
 v_{n-2j}(\zeta)
\left( \frac{\ol{\zeta-z}}{\zeta-z} \right) ^{j},\; z\in\OM.
\end{align}
As shown in \cite[Theorem 2.2]{sadiqTamasan02}, if $\bv=\langle v_{0}, v_{-1}, v_{-2},...\rangle  \in l^{1,1}_{\INF}(\Gam)\cap C^\mu(\Gam;l_1)$, then $\B \bv\in C^{1,\mu}(\OM;l_\infty)\cap C(\ol \OM;l_\infty)$  is $\mathcal{L}^2$-analytic in $\OM$. If  $\bv\in C^{1,\mu}(\OM;l_\infty)\cap C(\ol \OM;l_\infty)$ is $\mathcal{L}^2$-analytic in $\ol\OM$, then $\bv(z)=\B\bv(z)$, for $z\in\ol\OM$.

Also similar to the analytic maps, the traces on the boundary of $\mathcal{L}^2$-analytic maps satisfy some constraints, which can be expressed in terms of a corresponding Hilbert transform introduced in  \cite{sadiqTamasan01}. More precisely, the Bukhgeim-Hilbert transform $\HT$ is defined 
component-wise for $n \leq 0$ by
 \begin{align}\label{hilbertT}
(\HT\bv)_{n}(\xi)&=\frac{1}{\pi} \int_{\Gam } \frac{v_{n}(\zeta)}{\zeta - \xi} d\zeta 
+\frac{1}{\pi} \int_{\Gam } \left \{ \frac{d\zeta}{\zeta-\xi}-\frac{d \ol{\zeta}}{\ol{\zeta-\xi}} \right \} \sum_{j=1}^{\infty}  
v_{n-2j}(\zeta)
\left( \frac{\ol{\zeta-\xi}}{\zeta-\xi} \right) ^{j},\; \xi\in\Gam,
\end{align}
and we refer to \cite{sadiqTamasan01} for its mapping properties. For the proof of the theorem below we refer to  \cite[Theorem 3.2]{sadiqTamasan01}.

\begin{theorem}\label{BukhgeimCauchyThm}
Let $\bv = \langle v_{0}, v_{-1}, v_{-2},...\rangle\in l^{1,1}_{\INF}(\Gam)\cap C^\mu(\Gam;l_1)$ be defined on the boundary $\Gam$.
Then $\bv$ is the boundary value of an $\mathcal{L}^2$-analytic function if and only if
\begin{align} \label{NecSufEq}
   (I+\i\HT) \bv = {\bf {0}}.
\end{align} 
\end{theorem}

In addition to $\mathcal{L}^2$-analytic maps, another ingredient consists in the one-to-one relation between solutions
$ \bu: = \langle u_{0}, u_{-1},u_{-2},... \rangle $
to
\begin{align}\label{beltrami}
\dba u_{-n}(z) +\del u_{-n-2}(z) + a(z)u_{-n-1}(z) &=0,\quad z\in \OM, \; n\geq 0.
\end{align}
and the $\mathcal{L}^2$-analytic map $\bv$ satisfying
\begin{align}\label{Analytic}
 \dba v_{-n}(z) +\del v_{-n-2}(z) &=0,\quad z\in \OM, \; n\geq 0;
\end{align}see \cite[Lemma 4.2]{sadiqScherzerTamasan} for details. The relation can be expressed via the convolutions
\begin{equation}\label{conv_uandv}
\begin{aligned}
 v_{-n}(z) &= \sum_{j=0}^{\INF} \alpha_j(z) u_{-n-j}(z), \quad z\in \OM, \; n \geq 0, \\
 u_{-n}(z) &= \sum_{j=0}^{\INF} \beta_j(z) v_{-n-j}(z), \quad z\in \OM, \; n \geq 0,
\end{aligned} 
\end{equation} where $\alpha_j$'s and  $\beta_j$'s are the Fourier modes of  $e^{\mp h}$,
\begin{align}\label{ehEq}
  e^{- h(z,\btheta)} := \sum_{m=0}^{\INF} \alpha_{m}(z) e^{\i m \tta}, \quad e^{h(z,\btheta)} := \sum_{m=0}^{\INF} \beta_{m}(z) e^{\i m \tta}, \quad (z, \tta) \in \ol\OM \times \sph, 
\end{align}with $h$ defined by
\begin{align}\label{hDefn}
h(z,\btheta) := Da(z,\btheta) -\frac{1}{2} \left( I - \i H \right) Ra(z\cdot \btheta^{\perp}, \btheta^{\perp}).
\end{align} In the above formula, $\btheta^\perp$ is  orthogonal  to $\btheta$, 
$Da(z,\btheta) =\ds \int_{0}^{\INF} a(z+t\btheta)dt$ is the divergent beam transform of the attenuation $a$, 
$Ra(s, \btheta^{\perp}) = \ds \int_{-\INF}^{\INF} a\left( s \btheta^{\perp} +t \btheta \right)dt$ is the Radon transform of
the attenuation $a$, and $H$ is the (infinite) Hilbert transform 
\begin{align}\label{classical_Hilbert}
H f(x)=\frac{1}{\pi}\int_{-\INF}^{\INF}\frac{f(s)}{x-s}ds
\end{align}
taken in the first variable and evaluated at $s = z \cdotp \btheta^{\perp}$. 
The function $h$ appeared first in \cite{naterrerBook} and enjoys the crucial property of having vanishing negative Fourier modes. 
We refer to \cite[Lemma 4.1]{sadiqScherzerTamasan} for the properties of $h$ used in here.

\section{Reconstruction  in $\OM^+$ of a sufficiently smooth isotropic source $f$}\label{partialInv}

The method of reconstruction proposed here considers the operator $ [I - \i H_t]$, where $H_t$ is the finite  Hilbert transform 
\begin{align}\label{finite_Hilbert}
H_t g(x)=\frac{1}{\pi}\int_{-l}^{l}\frac{g(s)}{x-s}ds, \quad x \in (-l,l),
\end{align} with the integral understood in the sense of principal value. 
It is well-known (\cite{tricomi57}) that $\i H_t$ is a bounded operator in $L^2(-l,l)$ with spectrum $[-1,1]$, see \cite{koppelmanPincus58,okadaElliott91}. However, $1$ is not in the point spectrum, see \cite{widom60} for a proof based on a Riemann-Hilbert problem. The arguments below (from \cite[Proposition 2.1]{fujiwaraSadiqTamasan21}) use the unitary property in $L^2(\BR)$ of the (infinite) Hilbert transform $H$ in \eqref{classical_Hilbert}. We repeat them here for reader's convenience.
\begin{prop}\label{prop_finiteHilbert}In $L^2(-l,l)$,
$\ds Ker [I-\i H_t] = \{ 0\}$.
\end{prop}
\begin{proof}
Let $f\in L^2(-l,l)$ be extended by zero to the entire real line.
If $f \in Ker [I-\i H_t]$, then 
$|f| = |H_tf|=|Hf|$ a.e. in $(-l,l)$, and 
\begin{align*}
 \norm{f}_{L^2(-l,l)}^2 &= \norm{f}_{L^2(\BR)}^2 = \norm{Hf}_{L^2(\BR)}^2 
 = \int_{-l}^{l} |H_tf|^2 dx + \int_{|x| >l} |Hf(x)|^2 dx \\
 &= \int_{-l}^{l} |f|^2 dx + \int_{|x| >l} |Hf(x)|^2 dx.
\end{align*}The latter identity shows that $|Hf(x)|=0$, for $|x| >l$. Thus, for $|x|>l$,
\begin{align*}
 0 &=\int_{-l}^{l}\frac{f(s)}{x-s}ds = \frac{1}{x} \int_{-l}^{l} f(s) \sum_{j=0}^{\INF} \left(\frac{s}{x} \right)^jds =   \sum_{j=0}^{\INF} \frac{1}{x^{j+1}} \int_{-l}^{l} f(s) s^jds,
\end{align*}yielding $f$ orthogonal on any polynomial. By density of polynomials in $L^2(-l,l)$, $f=0$.
\end{proof}

Recall the boundary value problem \eqref{TransportScatEq1}:
\begin{equation}\label{TransportScatEq2}
\begin{aligned} 
&\btheta\cdot\nabla u(z,\btheta) +a(z) u(z,\btheta) - \int_{\sph} k(z,\btheta \cdot \btheta')u(z,\btheta') d\btheta' = f(z), \quad (z,\btheta)\in \OM\times\sph,\\
&u\lvert_{\Gamma_-}=0,
\end{aligned}
\end{equation}
for an isotropic source $f$ and attenuation coefficient $a$, and  with a scattering kernel $k$ of the type \eqref{particular_kM},
\begin{align}\label{k}  
k(z,\cos \tta)  = k_0(z) +2\sum_{n=1}^{M} k_{-n}(z) \cos (n \tta),
\end{align}
for some fixed integer $M \geq1$.

We assume that $a, k_0,k_{-1},...,k_{-M}\in C^{2}(\ol\OM)$ are such that the forward problem \eqref{TransportScatEq2} has a unique solution $u\in L^p(\Omega\times\sph)$ for any $f\in L^p(\OM)$, see Theorem \ref{analytic_Fredholm}. We also assume an \emph{unknown} source of a priori regularity $f \in W^{2,p}(\ol \OM)$, $p > 4$. According to Theorem \ref{u_reg_Wp} part (ii), $u\in C^{1,\mu}(\OM\times\sph)$ with $\mu>1/2$. In agreement with the physics model, the functions $a,f,k$ are further assumed real valued, so that the solution $u$ is also real valued.  Note that, since $k(z, \cos \tta)$ in \eqref{k} is both real valued and even in $\tta$, the coefficient $k_{-n}$ in \eqref{k} is the  $(-n)^{th}$ Fourier coefficient of $k(z,\cos(\cdot))$. Moreover $k_{-n}$ is real valued, and $k_n(z)=k_{-n}(z) =\frac{1}{2\pi} \int_{-\pi}^{\pi} k(z,\cos \theta ) e^{i n\theta}d\theta$.

Let $u(z,\btheta) = \sum_{-\infty}^{\infty} u_{n}(z) e^{in\tta}$ be the formal Fourier series representation of the solution of \eqref{TransportScatEq2} in the angular variable $\btheta=(\cos\tta,\sin\tta)$. Since $u$ is real valued, $u_{-n}=\ol{u_n}$ and the angular dependence is completely determined by the sequence of its nonpositive Fourier modes 
\begin{align}\label{boldu}
\OM \ni z\mapsto  \bu(z)&: = \langle u_{0}(z), u_{-1}(z),u_{-2}(z),... \rangle.
\end{align}

Consider the decomposition of the advection operator $\btheta \cdot\nabla=e^{-\i \tta}\dba + e^{\i \tta}\del$, where $\dba = (\del_x+\i \del_y)/2$ and $\del =(\del_x-\i \del_y)/2$ are  derivatives in the spatial domain. By identifying the Fourier coefficients of the same order, the equation \eqref{TransportScatEq2} reduces to the system:
\begin{align}\label{source_syseq}
\overline{\del} u_{1}(z)+\del u_{-1}(z) + a(z)u_{0}(z)&=k_{0}(z)u_{0}(z)+f(z), \\ \label{finsys}
\dba u_{-n}(z) +\del u_{-n-2}(z) + a(z)u_{-n-1}(z) &= k_{-n-1}(z)u_{-n-1}(z),\quad 0 \leq n\leq M-1, \\ \label{infinite_sys}
\dba u_{-n}(z) +\del u_{-n-2}(z) + a(z)u_{-n-1}(z) &=0,\qquad n\geq M.
\end{align}

Without loss of generality we consider cartesian coordinates such that $\Lambda$ lies in the upper  half plane with endpoints on the real axis, and let $L=(-l,l)$ be the segment joining the endpoints of the arc. Let  $\OM^+=\{z\in\OM:\; \im{z}>0\}$ denote the convex hull of $\Lambda$,  and note that $\partial\OM^+=\Lambda\cup L$.

To simplify the statement of the next result, for each $n\geq 0$, let us introduce the functions $F_{-n}(z)$ defined for $z\neq\pm l$ in $\Lambda\cup L$ by
\begin{align}\label{Fanalytic}
F_{-n}(z):= \frac{1}{\i \pi }\int_\Lambda\frac{u_{-n}(\zeta)}{\zeta-z}d\zeta 
+\frac{1}{\i \pi }\int_\Lambda\left\{\frac{d\zeta}{\zeta-z}-\frac{d\ol\zeta}{\ol\zeta-\ol z}\right\}\sum_{j=1}^\infty u_{-n-2j}(\zeta)\left(\frac{\ol\zeta-\ol z}{\zeta-z}\right)^j.
\end{align}For $z\in\Lambda$,  the first integral is in the sense of principal value. Note that $F_{-n}$ is directly determined by the data $u_{-n}\lvert_\Lambda$,  $n\geq 0$. 


The proof of the following result is constructive and provides the basis of the reconstruction method implemented in Section \ref{Sec:numerics}. 
\begin{theorem}\label{main}
Let $\OM\subset\mathbb{R}^2$ be a strictly convex bounded domain, $\Lambda$ be an arc of its boundary,  and  $\Omega^+$ be the convex hull of $\Lambda$.  Consider the boundary value problem \eqref{TransportScatEq2} for some known real valued $a, k_0,k_{-1},...,k_{-M}\in C^{2}(\ol\OM)$ such that \eqref{TransportScatEq2} is well-posed. If the unknown source $f$ is real valued and $W^{2,p}(\ol \OM)$-regular, with $p > 4$, then $u\lvert_{\Lambda_+}$ uniquely determines $f$ in $\OM^+$. 
\end{theorem}
\begin{proof}

Let $u$ be the solution of the boundary value problem \eqref{TransportScatEq2} and  $\bu= \langle u_{0}, u_{-1}, u_{-2}, ... \rangle $ 
be the sequence valued map of its non-positive Fourier modes, 
Since $f \in W^{2,p}(\ol \OM)$, $p > 4$, then by Proposition \ref{u_reg_Wp} (ii), $u \in W^{2,p}(\Omega\times\sph)$. By the Sobolev embedding, $u\in C^{1,\mu}(\ol \OM\times\sph)$ with 
$\mu =1-\frac{2}{p}>\frac{1}{2}$, and thus, by \cite[Proposition 4.1 (i)]{sadiqTamasan01},  $\bu \in l^{1,1}_{\INF}(\del \OM^+)\cap C^\mu(\del \OM^+;l_1)$. 

We note that the shifted sequence valued map $\mathcal{L}^M\bu$ solves 
\begin{align}\label{Lmu_beltrami}
\dba \mathcal{L}^{M} \bu(z) +\mathcal{L}^2 \del \mathcal{L}^{M} \bu(z)+ a(z)\mathcal{L}^{M+1}\bu(z) = \bzero,\quad z\in \OM,
\end{align}and then
the associated sequence valued map $\mathcal{L}^M\bv =(v_{-M},v_{-M-1},v_{-M-2}...)$ defined by the convolutions \eqref{conv_uandv}  solves 
\begin{align}\label{vn_infinite_sysM}
 \dba v_{-n}(z) +\del v_{-n-2}(z) &=0,\quad z\in \OM, \; n\geq M.
\end{align} 
In particular, $\mathcal{L}^M\bv$ is $\mathcal{L}^2$-analytic.

By \eqref{data}, the data $u|_{\Lambda_+} = g$ on $\Lambda_+$ determines $\mathcal{L}^M\bu$ on $\Lambda$. By the convolution formula \eqref{conv_uandv} for $n\geq M$, $\mathcal{L}^M\bu \lvert_{\Lambda}$  determines the traces $\mathcal{L}^M\bv \in l^{1,1}_{\INF}(\Lambda)\cap C^\mu(\Lambda;l_1)$ on $\Lambda$.

Since $\mathcal{L}^M\bv  \in l^{1,1}_{\INF}(\del \OM^+)\cap C^\mu(\del \OM^+;l_1)$ is the boundary value of an $\mathcal{L}^2$-analytic function in $\OM^+$, then the necessity part of Theorem \ref{BukhgeimCauchyThm} yields 
\begin{align}\label{Cond1}
[I+\i\HT] \mathcal{L}^M\bv={\bf {0}},
\end{align}where $\HT$ is the Bukhgeim-Hilbert transform in \eqref{hilbertT}.

 We consider \eqref{Cond1} on $L=(-l,l)$, where for each $x\in (-l,l)$ and $n\geq M$, the $n$-th component yields
\begin{align}\nonumber
v_{-n}(x)  - \frac{\i}{\pi} \int_{-l}^{l}  \frac{v_{-n}(s)}{x-s}  ds = -\frac{\i}{\pi} \int_{\Lambda} \frac{v_{-n}(\zeta)}{\zeta-x} d\zeta 
 &-\frac{\i}{\pi} \int_{\Lambda } \left \{ \frac{d\zeta}{\zeta-x}-\frac{d \ol{\zeta}}{\ol{\zeta}-x} \right \}
  \sum_{j=1}^{\infty}  v_{-n-2j}(\zeta)
\left( \frac{\ol{\zeta}-x}{\zeta-x} \right) ^{j} \\
&-
\frac{\i}{\pi} 
\int_{-l}^{l} \left \{ \frac{d\zeta}{\zeta-x}-\frac{d \ol{\zeta}}{\ol{\zeta}-x} \right \}
  \sum_{j=1}^{\infty}  v_{-n-2j}(\zeta)
\left( \frac{\ol{\zeta}-x}{\zeta-x} \right) ^{j}
.\label{knowntobezero}
\end{align}

Since the last integral in \eqref{knowntobezero} ranges over the reals, it vanishes. The remaining two integrals in the right hand side give $F_{-n}(x)$  in \eqref{Fanalytic}, and \eqref{knowntobezero} becomes
\begin{align}\label{Pminus_un}
[I - \i H_t](v_{-n})(x)=& F_{-n}(x), \quad x \in L, \quad n \geq M,
\end{align}where $H_t$ is the finite  Hilbert transform in \eqref{finite_Hilbert}. For each $n\geq M$, by Proposition \ref{prop_finiteHilbert},  $v_{-n} \big \lvert_{L}$ is determined as the unique solution in $L^2(-l,l)$ of \eqref{Pminus_un}. 

Note that the equation \eqref{Pminus_un} may not have any solution for an arbitrary right hand side in $L^2(-l,l)$. However, in our inverse problem, the function $F_{-n}$ already belongs to the range of $I -\i H_t$, so that the solution exists. Moreover, since the range is open, \eqref{Pminus_un} is uniquely solvable in a sufficiently small  $L^2$-neighborhood of $F_{-n}$.

With $v_{-n}$ now known on $\Lambda\cup L$ for $n \geq M$, we apply the Bukhgeim-Cauchy integral formula \eqref{BukhgeimCauchyFormula} to find $v_{-n}$ for $n \geq M$  in $\OM^+$. 

Using again the convolution formula \eqref{conv_uandv}, now in  $\OM^+$, we determine $u_{-n}$ for $n \geq M$  inside $\OM^+$. In particular we recovered $u_{-M-1}, u_{-M}$.

Recall that $u_0, u_{-1}, u_{-2}, \cdots, u_{-M}, u_{-M-1}$ satisfy
\begin{subequations} \label{Dirichlet_inhomDbar}
\begin{align} \label{dbaU_eq}
\dba u_{-M+j} &= -\del u_{-M+j-2} - \left[ (a- k_{-M+j-1}) u_{-M+j-1} \right], \quad 1\leq j \leq M,\\ 
\label{u_Lambda}  u_{-M+j} \lvert_{\Lambda} &= g_{-M+j}.
\end{align}
\end{subequations} 

We solve \eqref{Dirichlet_inhomDbar} iteratively for $j=1,2,...,M$, as a Cauchy problem for the $\dba$-equation with partial boundary data on $\Lambda$, 
\begin{subequations} \label{Cauchyproblem}
\begin{align} \label{dba_Cauchyproblem}
\dba w &= \Psi, \quad \text{in} \; \OM^+,\\ \label{u_Lambda_data}
     w &= g \quad \text{on} \; \Lambda,
\end{align}
\end{subequations} 
via the Cauchy-Pompeiu formula \cite{vekua_book62}:
\begin{align}\label{CauchyPompeiuformula}
w(z)=\frac{1}{2\pi \i}\int_{\del \OM^+} \frac{w(\zeta)}{\zeta-z}d \zeta -  
\frac{1}{\pi} \iint_{\OM^+} \frac{\Psi(\zeta)}{\zeta-z} d \xi d \eta, 
\quad \zeta = \xi + \i \eta, \quad z \in \OM^+.
\end{align}

For $\Psi\in L^p(\Omega)   $, $p>2$, and $g\in L^p(\Lambda)$, any $w$ defined by \eqref{CauchyPompeiuformula} solves \eqref{dba_Cauchyproblem}. However, for the boundary condition  \eqref{u_Lambda_data} to be satisfied 
the following compatibility condition needs to hold: By taking the limit $\ds\OM^+ \ni z \rightarrow z_0 \in \del\OM^+$ in \eqref{CauchyPompeiuformula} and using the Sokhotski-Plemelj formula \cite{muskhellishvili} in the first integral, and the continuous dependence on $z$ of the area integral \cite[Theorem 1.19]{vekua_book62}, the trace $w|_{\del\OM^+}$ and $\Psi$ must satisfy
\begin{align*}
w(z_0)=\frac{1}{2\pi \i}\int_{\del \OM^+} \frac{w(\zeta)}{\zeta-z_0}d \zeta + \frac{1}{2} w(z_0) -  \frac{1}{\pi} \iint_{\OM^+} \frac{\Psi(\xi,\eta)}{(\xi-z_0) +\i \eta} d \xi d \eta, \quad z_0 \in \del\OM^+.
\end{align*} 
In our inverse problem this compatibility condition is already satisfied for $z_0\in \Lambda$. We use this compatibility condition for $z_0\in L$, to recover the missing boundary data $w|_L$.
More precisely, by Proposition  \ref{prop_finiteHilbert}, $w|_L$ is the unique solution of
\begin{align}\label{Hilbert_Pompeiuformula}
[I - \i H_t]w(z_0) =\frac{1}{\pi \i}\int_{\Lambda} \frac{g(\zeta)}{\zeta-z_0}d \zeta -  \frac{2}{\pi} \iint_{\OM^+} \frac{\Psi(\xi,\eta)}{(\xi-z_0) +\i \eta} d \xi d \eta, \quad z_0 \in L.
\end{align}

If $g\in L^p(\Lambda)$, $p\geq2$, then  \eqref{Hilbert_Pompeiuformula} provides a unique solution $w|_L\in L^p(-l,l)$. Moreover, for $\Psi\in L^p(\OM)$, $p>2$, the solution $w$  of \eqref{dba_Cauchyproblem} is provided by \eqref{CauchyPompeiuformula} and lies in $W^{1,p}(\OM)$. In the iteration, the right hand side of \eqref{dbaU_eq} is again in $L^p(\OM)$, and the iteration can proceed.

We solve repeatedly \eqref{Dirichlet_inhomDbar} for $j=1,....M$ to  recover $u_{-1}$ and $u_0$ in $\OM^+$. A priori, from the regularity of the solution of the forward problem, we known that $u_{-1}, u_{0}\in C^{1,\mu}(\ol\OM)$ so that the source $f$ is recovered pointwise by 
\begin{align} \label{fsource_Scatpoly}
 f|_{\OM^+}(z) = 2 \re \left(\del u_{-1}(z)\right) + \left( a(z)-k_0(z) \right) u_0(z),\quad z\in\OM^+
\end{align}as a $C^\mu(\ol{\OM^+})$-map.
\end{proof}


We summarize below in a stepwise fashion the reconstruction of $f$ in the convex hull ${\OM^+}$ of the boundary arc $\Lambda$. Recall that $L$ is the segment joining the endpoints of $\Lambda$.


\subsection*{Reconstruction procedure}
Consider the data $\bu \lvert_{\Lambda}$.
\begin{enumerate}
\item Using formula \eqref{conv_uandv}, and data $\mathcal{L}^M\bu$ on $\Lambda$, determine the traces $\mathcal{L}^M\bv \lvert_{\Lambda}$ on $\Lambda$.
 \item Recover the traces $\mathcal{L}^M\bv\lvert_{L}$ pointwise on $L$ as follows:
      \begin{enumerate}
               \item Using $\mathcal{L}^M\bv \lvert_{\Lambda}$, compute by formula \eqref{Fanalytic},  the   
                          function $F_{-n}$, for each $n\geq M$.
       \item Recover for each $n\geq M$, the trace $v_{-n} \big \lvert_{L}$ by solving \eqref{Pminus_un}.
      \end{enumerate}
 \item By Bukhgeim-Cauchy formula \eqref{BukhgeimCauchyFormula}, extend $v_{-n}$ for $n \geq M$ from the boundary $\Lambda\cup L$ to $\OM^+$.
\item Using again formula \eqref{conv_uandv}, now in  $\OM^+$, recover $u_{-n}$ for $n \geq M$  inside $\OM^+$. 
 \item Using $u_{-M-1}, u_{-M}$, recover the modes $u_{-M+1}, u_{-M+2}, \cdots, u_{-1}, u_0 $ recuresively as follows:
 \begin{enumerate}
		\item Using data $u_{-M+1} \lvert_{\Lambda}$, recover the trace $u_{-M+1} \big \lvert_{L}$ by solving \eqref{Hilbert_Pompeiuformula}.
		\item Using $u_{-M+1} \lvert_{\Lambda \cup L}$, recover $u_{-M+1} $ inside $\OM^+$  by the  Cauchy-Pompeiu formula \eqref{CauchyPompeiuformula}. 
		\item Now iterate the steps (5 a) and (5 b) to find the modes $ u_{-M+2}, \cdots, u_{-2}, u_{-1}, u_0 $ in $\OM^+$. 
\end{enumerate}
\item Recover $f|_{\OM^+}$ by formula \eqref{fsource_Scatpoly}.
\end{enumerate}


\section{Numerical results}\label{Sec:numerics}

To illustrate the numerical feasibility of the proposed method and its extensibility to general settings, in this section we present the results of two numerical experiments. The rigorous analysis on the numerical methods employed below requires further study and is left for a separate discussion.

The domain $\OM$ is the unit disk, the measurement boundary $\Lambda$ is the upper semicircle, and $\OM^+$ is the upper semidisk. The numerical experiments consider the boundary value problem  \eqref{TransportScatEq2}
with the attenuation 
\[
a(z) = \mu_{\text{s}}(z) + \mu_{\text{a}}(z),
\]
where $\mu_{\text{s}}$,  and $\mu_{\text{a}}$ are the scattering and absorption coefficients, respectively.

In the first numerical experiment $a \in C^2(\ol\Omega)$ and the scattering kernel is homogeneous with
\begin{equation}\label{quadratic_scat}
 k(\btheta\cdot\btheta') = \dfrac{ \mu_{\text{s}}}{2\pi}\left[ 1+2g(\btheta\cdot\btheta')+ 2g^2\cos(2\arccos (\btheta\cdot\btheta'))\right].
\end{equation}

In contrast, the second numerical experiment uses a discontinuous absorption coefficient $\mu_a$,  and a homogeneous scattering kernel $k$ of infinite Fourier content. More precisely, we work with 
the two dimensional Henyey-Greenstein (Poisson) kernel
\begin{equation}\label{poisson_scat}
k(\btheta\cdot\btheta')
= \mu_{\text{s}}\dfrac{1}{2\pi}\dfrac{1-g^2}{1-2g\btheta\cdot\btheta' + g^2}.
\end{equation}
Note that \eqref{quadratic_scat} comes from the quadratic truncation of \eqref{poisson_scat}.

Throughout this section, $\mu_{\text{s}}\equiv 5$ and $g=1/2$ are used,
and computations are processed in the standard double precision arithmetic.
The parameter $g=1/2$ yields an anisotropic scattering half way between the ballistic ($g=0$) regime and an isotropic scattering regime ($g=1$). 
The value of the parameter $\mu_{\text{s}}$ yields that, on average, a  particle scatters after running straight every $1/5$ of the unit pathlength. These parameter choices are meaningful in certain optical regimes, where the diffusion approximation would not hold.

Let $R=(-0.25,0.5) \times (-0.15,0.15)$ be a rectangular, and
\begin{align*}
B_1 &= \{ (x,y) \::\: (x-0.5)^2 + y^2 < 0.3^2 \},\\
B_2 &= \left\{ (x,y) \::\: (x+0.25)^2 + \left(y-\dfrac{\sqrt{3}}{4}\right) < 0.2^2 \right\}, \quad \text{and}\\
B_3 &= \{ (x,y) \::\: x^2 + (y+0.6)^2 < 0.3^2 \}
\end{align*}
be circular regions inside $\OM$ as illustrated in Figure \ref{fig:domain}.
\begin{figure}[t]
\centering
\begin{tikzpicture}[scale=2]
\draw [line width=2] (0,0) circle (1); 
  \node at (0.9,0.9) {$\Omega$};

\draw [color=gray,fill=gray,line width=0] (-0.25,0.433013) circle (0.2); 
\draw [color=gray,fill=gray,line width=0] (0,-0.6) circle (0.3); 
  \node at (-0.4,-0.4) {$B_3$};

\draw [color=gray,fill=gray,line width=0] (-0.25,-0.15) rectangle (0.5,0.15); 
  \node at (-0.4,-0.1) {$R$};

\draw [dotted,line width=2] (0.5,0) circle (0.3); 
  \node at (0.5,0.45) {$B_1$};

\draw [dotted,line width=2] (-0.25,0.433013) circle (0.2); 
    \node at (-0.25,0.8) {$B_2$};

\draw [->] (-1.2,0) -- (1.2,0);
\draw [->] (0,-1.2) -- (0,1.2);

\end{tikzpicture}

\caption{\label{fig:domain}Domain $\OM$ and inclusions; the dotted circles ($B_1$ and $B_2$)
indicates highly absorbing regions, while gray regions ($B_2$, $R$ and $B_3$)
are support of internal sources.}
\end{figure}
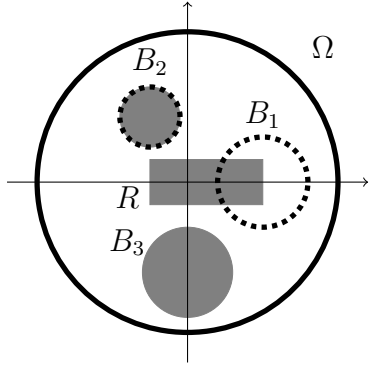
The source term
\[
f(z) = \begin{cases}
2, \quad & \text{in $R$};\\
1, \quad & \text{in $B_2\cup B_3$};\\
0, \quad & \text{otherwise},
\end{cases}
\]used to generate the boundary data on $\Lambda_+$ is to be reconstructed  in the upper semidisk $\OM^+$.

In the first example
the absorption coefficient is a $C^2$-smoothen version of the discontinuous case in \eqref{disc_a}. Namely, for $\epsilon = 0.05$,
\[
\mu_{\text{a}}(z) = \begin{cases}
2, \quad & \text{in $\{ z \::\: \dist(z,\partial B_1) \ge \epsilon \}\cap B_1$};\\
1, \quad & \text{in $\{ z \::\: \dist(z,\partial B_2) \ge \epsilon \}\cap B_2$};\\
0.1, \quad & \text{in $\{ z \::\: \dist(z,B_1) \ge \epsilon \}\cap \{ z \::\: \dist(z,B_2) \ge \epsilon \}$},
\end{cases}
\]
while in the $\epsilon$-neighborhoods of $\partial B_1$ and $\partial B_2$  we use some translations and scalings of the quintic polynomial $-(|z|-1)^3(6|z|^2+3|z|+1)$, $|z|\leq 1$, to define $\mu_{\text{a}}\in C^2(\ol\OM)$.

For our inverse problem,
the ``measured" data \eqref{data} is generated by the numerical computation
of the corresponding forward problem~\eqref{TransportScatEq1} in $\OM\times\sph$, where we retain the trace of the solution $u|_{\Lambda_+}$ and disregard the rest. Note the contribution to the data of radiation coming from the source supported in the lower half of the rectangle $R$ and from the ball $B_3$. 

The numerical solution is processed by the piecewise constant approximation method in \cite{fujiwara},
where the spatial domain $\OM$ and $\sph$ are divided into
$4,823,822$ triangles (the maximum diameter is approximately $0.0025$)
and $360$ velocities at equi-intervals respectively.
The boundary $\partial\OM$ is approximated by $5,234$ equi-length segments and,
thus, $2,617$ spatial measurement nodes are assigned on $\Lambda$.

The computed data is depicted in Figure~\ref{fig:measurement}.  In there, for  $\zeta\in\Lambda$  (indicated by $\times$), the red closed curves represent $\zeta + 2u(\zeta,\btheta)$.
The zero incoming flux boundary condition can be observed on the computed radiation, where the curves do not enter $\OM$.
\begin{figure}[h]
\begin{minipage}{.5\textwidth}
\centering
\includegraphics[bb=220 140 820 680, width=.8\textwidth]{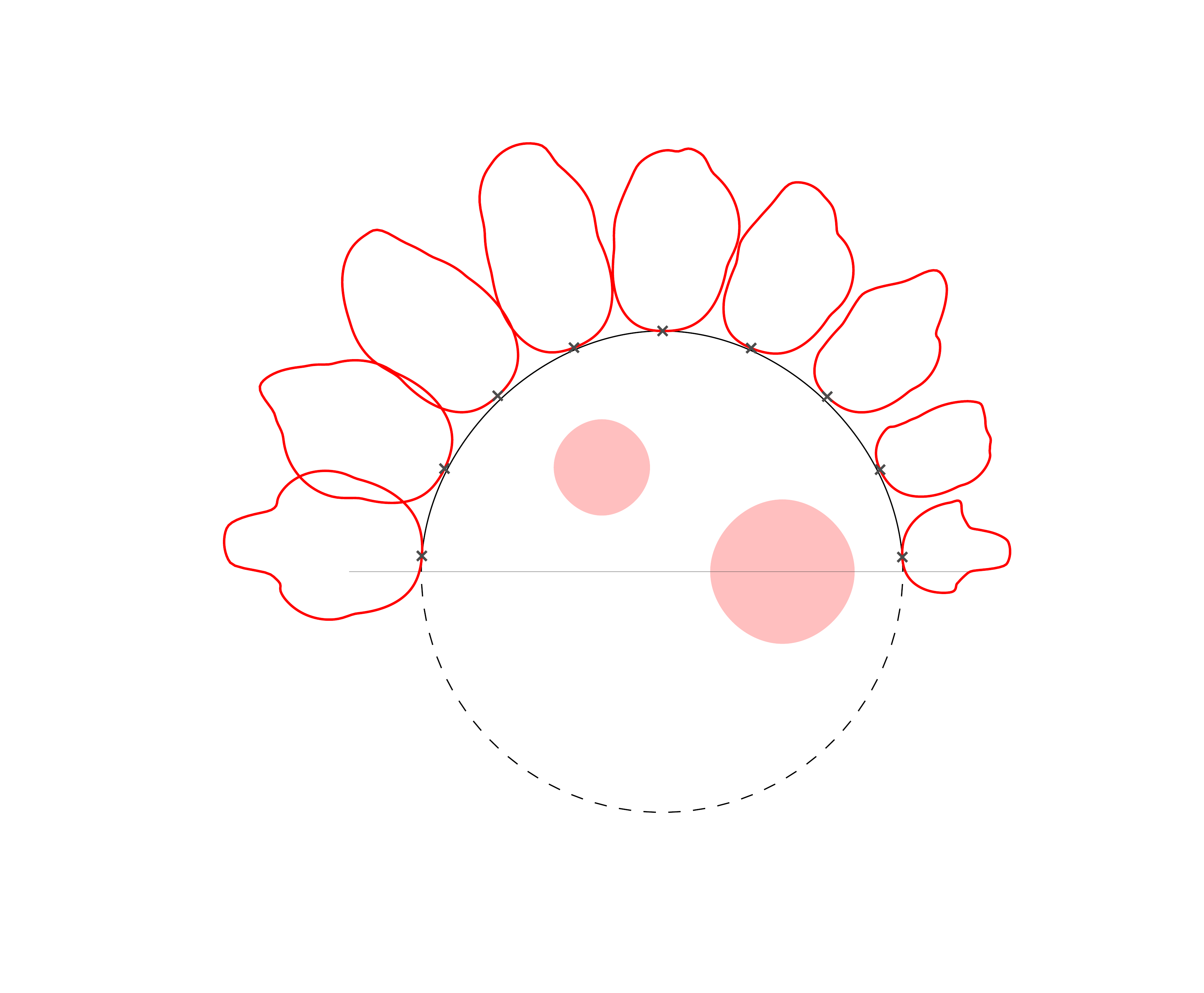}
\end{minipage}
\begin{minipage}{.3\textwidth}
\includegraphics[width=.7\textwidth]{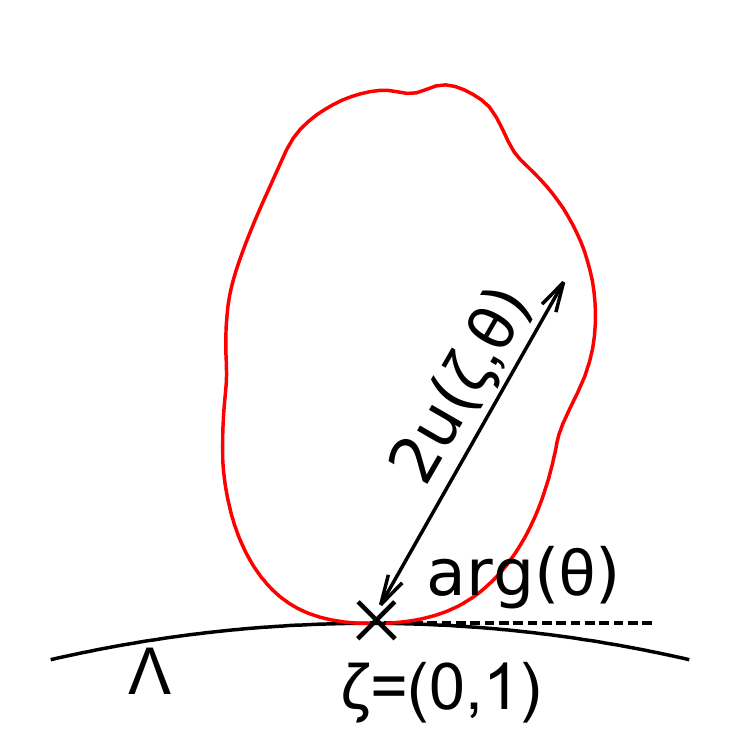}
\end{minipage}
\caption{\label{fig:measurement}Boundary measurement $u|_{\Lambda_+}$;
(left) the measurement arc $\Lambda$ lies in the upper half plane.
Each red closed curve stands for $\zeta+2u(\zeta,\btheta)$ for $\btheta \in \sph$
at $\zeta \in \Lambda$ indicated by the cross symbol ($\times$);
(right) magnification of $u(\zeta,\btheta)$ at $\zeta = (0,1)$.
}
\end{figure}

In our numerical reconstruction,
the domain of interest $\OM^+$ is partitioned into a triangular mesh without any prior information on $R$ and $B_i$, $i=1,2,3$.
The number of triangles ($8,631$) in this mesh is much less than that the number  of triangles ($4,823,822$) used in the  computation in the forward problem, thus avoiding an inverse crime.
The attenuation coefficient $a$ is assumed known in $\ol\OM$. The triangular mesh induces $157$ nodes on $\Lambda$ and $100$ nodes on $L$.

The Hilbert transform in \eqref{hDefn} is computed by a method proposed in \cite{fujiwaraSadiqTamasan20}.
All the integrations are approximated by the composite mid-point rule
with equi-spaced intervals. The integrating factors in \eqref{ehEq} are computed with $100$ subintervals and $360$ velocities,
and the integral equations~\eqref{Pminus_un} and \eqref{Hilbert_Pompeiuformula} are computed with
$1,666$ subintervals (so that $\Delta x = 2/1666\approx 0.0012$ on $L$ is about the same with the length $\pi/2617$ of the partition on the arc $\Lambda$).
Of particular interest, and key to our procedure, is the numerical computation of the  integral equations~\eqref{Pminus_un} and \eqref{Hilbert_Pompeiuformula},
which is done via the collocation method with the numerical integration rule
\[
[I - \i H_t]u(x_i)
\approx u(x_i) - \dfrac{\i}{\pi}\sum_{j\neq i}\dfrac{u(x_j)}{x_i - x_j}\Delta x.
\]

Except for the implicit regularization due to the discretization, no other regularization method is  explicitly employed in numerical implementation. The numerical reconstruction takes approximately $115$ seconds
by OpenMP parallel computation
on two Xeon E5-2650 v4 (2.20GHz, 12 cores) processors.

Figure ~\ref{fig:smooth} depicts numerical reconstruction results:
on the left is the profile of reconstructed $f(z)$,
and on the right is its section on the dotted line ($y = -\sqrt{3}x$).
In the figure the support of $f(z)$ is successfully reproduced. The reconstructed source $f$ away from the segment $L$ is
quantitatively in fine agreement with the exact one. However the accuracy of the reconstruction decreases at points on the support of $f$ which are close to the line segment $L$. 
One of the reasons is the ill-posedness in equation \eqref{Pminus_un}, which we currently mitigate solely via the discretization.  For a more accurate numerical reconstruction at such points, the choice of a better regularizer is needed. 

\begin{figure}[h]
\begin{minipage}{.65\textwidth}
\includegraphics[bb=75 50 355 185,width=\textwidth]{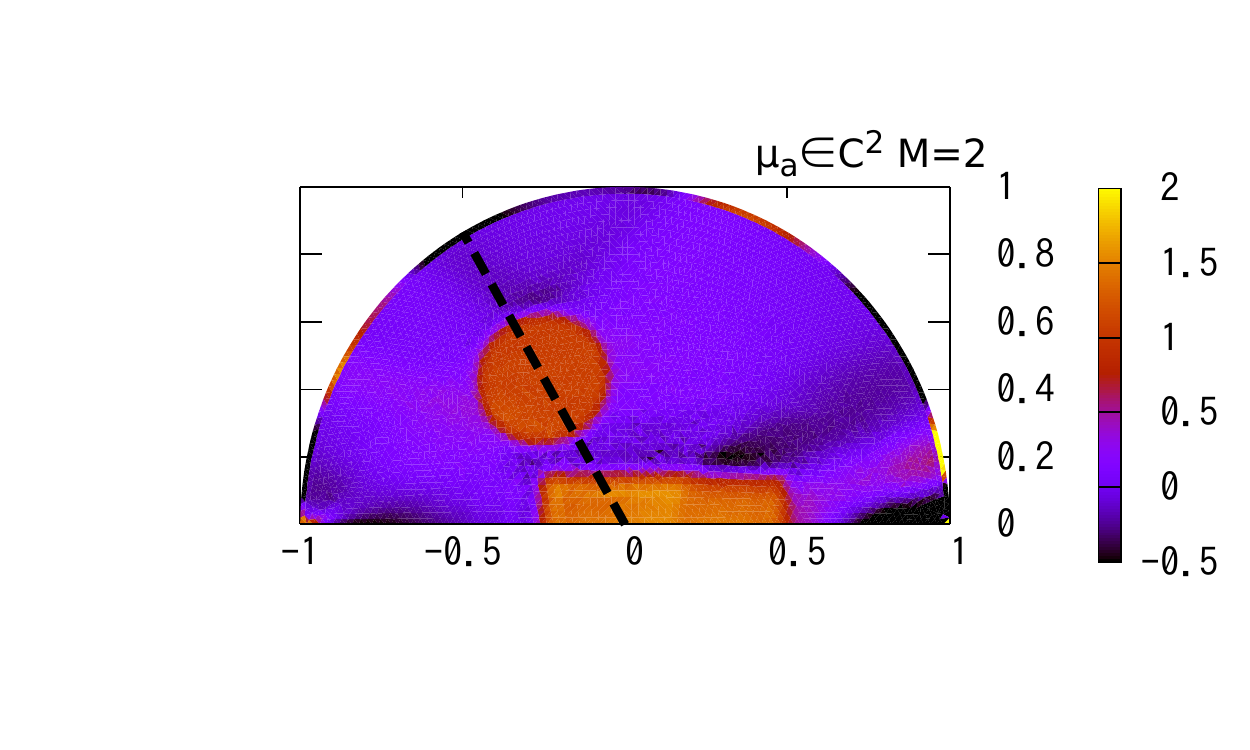}
\end{minipage}
\begin{minipage}{.3\textwidth}
\includegraphics[bb=55 0 190 150,width=\textwidth]{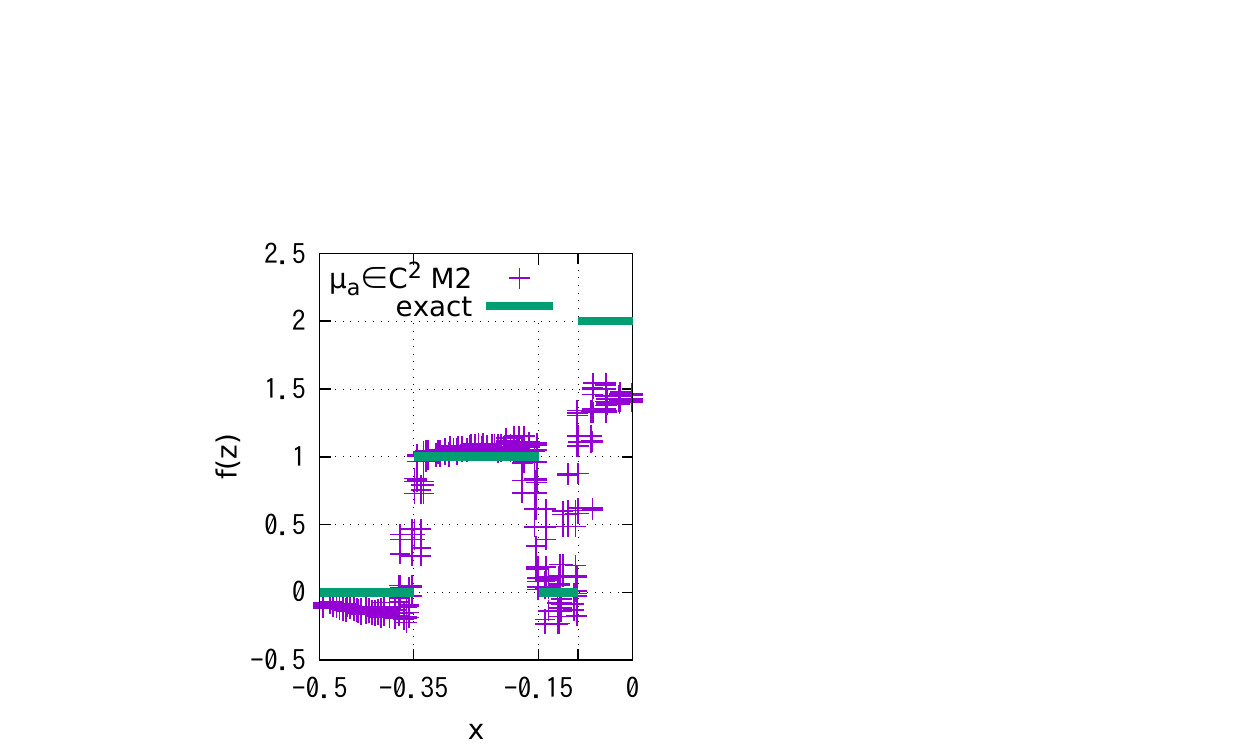}
\end{minipage}
\caption{\label{fig:smooth}Numerically reconstructed $f(z)$ for
polynomial type scattering kernel \eqref{quadratic_scat} and $\mu_{\text{a}}\in C^2(\OM)$.
(left) the profile of reconstructed $f$ in the domain of interest $\OM^+$;
(right) the section on the dotted line.}
\end{figure}

In the second example, the proposed algorithm is applied to 
the non-polynomial type scattering kernel \eqref{poisson_scat} and the discontinuous absorption coefficient
\begin{align}\label{disc_a}
\mu_{\text{a}}(z) = \begin{cases}
2, \quad & \text{in $B_1$};\\
1, \quad & \text{in $B_2$};\\
0.1, \quad & \text{otherwise}.
\end{cases}
\end{align}

\begin{figure}[h]
\begin{minipage}{.65\textwidth}
\includegraphics[bb=75 50 355 185,width=\textwidth]{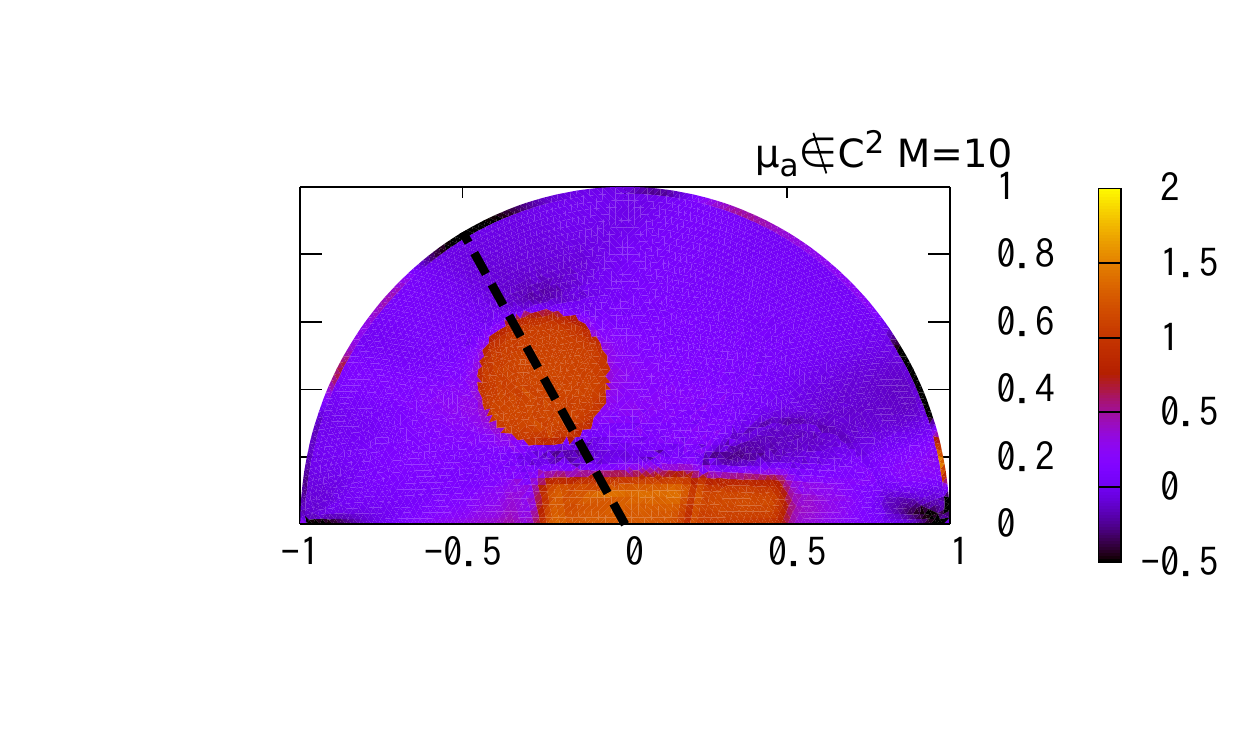}
\end{minipage}
\begin{minipage}{.3\textwidth}
\includegraphics[bb=55 0 190 150,width=\textwidth]{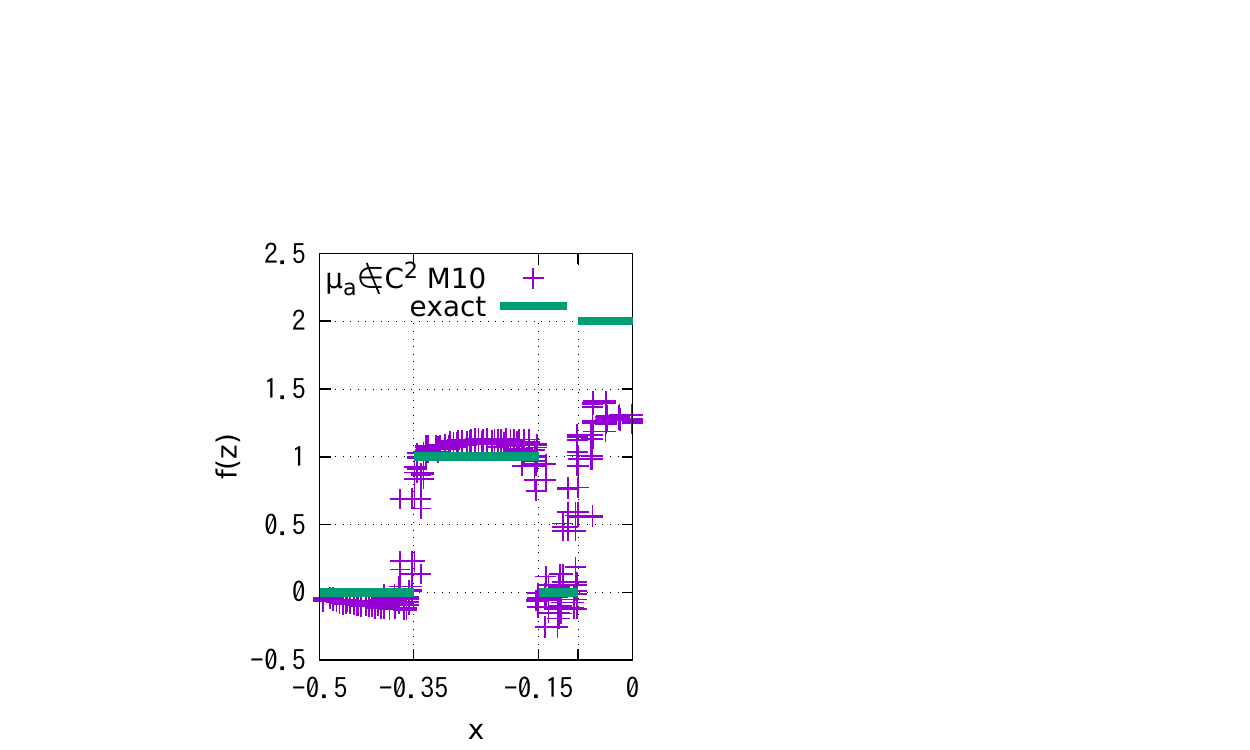}
\end{minipage}
\caption{\label{fig:discont}Numerically reconstructed $f(z)$ for the case with
Poisson kernel (non-polynomial type scattering kernel) and discontinuous $\mu_{\text{a}}$.}
\end{figure}

Figure ~\ref{fig:discont} shows numerical results,
where \eqref{Dirichlet_inhomDbar} is recursively solved 10 times ($M=10$). Even though the smoothness hypotheses in Theorem~\ref{main} are violated, and despite of a scattering kernel of infinite Fourier content, the numerical results are in agreement with the results in the first example (where all the hypotheses in Theorem~\ref{main} were satisfied). This illustrates the robustness of the proposed reconstruction method, and indicates that the reconstruction result in Theorem~\ref{main} may hold under more relaxed hypotheses.


\section*{Acknowledgment}
The work of H.~ Fujiwara was supported by JSPS KAKENHI Grant Numbers JP19H00641 and JP20H01821.
The work of K.~ Sadiq  was supported by the Austrian Science Fund (FWF), Project P31053-N32. 
The work of A.~Tamasan  was supported in part by the NSF grant DMS-1907097.


\end{document}